\documentclass[11pt,english]{article}
\usepackage[latin9]{inputenc}
\usepackage{geometry}
\usepackage{color}
\usepackage{babel}
\usepackage{array}
\usepackage{amsmath}
\usepackage{amsthm}
\usepackage{amssymb}
\usepackage[unicode=true,pdfusetitle,
 bookmarks=false,
 breaklinks=false,pdfborder={0 0 1},backref=page,colorlinks=true]
 {hyperref}
\hypersetup{
 citecolor=blue, linkcolor=blue}

\makeatletter

\newcommand{\noun}[1]{\textsc{#1}}
\providecommand{\tabularnewline}{\\}

\numberwithin{equation}{section}
\numberwithin{figure}{section}
\theoremstyle{plain}
\newtheorem{thm}{\protect\theoremname}[section]
\theoremstyle{definition}
\newtheorem{example}[thm]{\protect\examplename}
\theoremstyle{definition}
\newtheorem{defn}[thm]{\protect\definitionname}
\theoremstyle{plain}
\newtheorem{conjecture}[thm]{\protect\conjecturename}
\theoremstyle{remark}
\newtheorem{rem}[thm]{\protect\remarkname}
\theoremstyle{plain}
\newtheorem{lem}[thm]{\protect\lemmaname}
\theoremstyle{plain}
\newtheorem{prop}[thm]{\protect\propositionname}
\theoremstyle{plain}
\newtheorem{cor}[thm]{\protect\corollaryname}

\usepackage{graphicx}
\usepackage{tocloft}
\usepackage{xcolor}
\usepackage{enumitem}
\setlist{leftmargin=\parindent,align=left,labelwidth=\parindent,labelsep=-0.3em}

\theoremstyle{plain}
\newtheorem{ques}[thm]{Question}
\makeatother

\providecommand{\conjecturename}{Conjecture}
\providecommand{\corollaryname}{Corollary}
\providecommand{\definitionname}{Definition}
\providecommand{\examplename}{Example}
\providecommand{\lemmaname}{Lemma}
\providecommand{\propositionname}{Proposition}
\providecommand{\remarkname}{Remark}
\providecommand{\theoremname}{Theorem}

\begin{document}
\global\long\def\cay{\mathrm{Cay}}%
\global\long\def\sch{\mathrm{Sch}}%
\global\long\def\schone{\mathrm{Sch}\left(S_{n}\curvearrowright\left[n\right],\Sigma\right)}%
 
\global\long\def\schk#1{\sch\left(S_{n}\curvearrowright\left[n\right]_{#1},\Sigma\right) }%
 
\global\long\def\std{\mathrm{std}}%
 
\global\long\def\sgn{\mathrm{sgn}}%
 
\global\long\def\triv{\mathrm{triv}}%
 
\global\long\def\supp{\mathrm{supp}}%
 
\global\long\def\T{\mathfrak{eight}}%
 
\global\long\def\x{{\cal \tilde{\chi}}}%
 
\global\long\def\defi{\stackrel{\mathrm{def}}{=}}%

\title{Aldous' Spectral Gap Conjecture for Normal Sets}
\author{Ori Parzanchevski and Doron Puder}
\maketitle
\begin{abstract}
Let $S_{n}$ denote the symmetric group on $n$ elements, and $\Sigma\subseteq S_{n}$
a symmetric subset of permutations. Aldous' spectral gap conjecture,
proved by Caputo, Liggett and Richthammer \cite{caputo2010proof},
states that if $\Sigma$ is a set of transpositions, then the second
eigenvalue of the Cayley graph $\cay\left(S_{n},\Sigma\right)$ is
identical to the second eigenvalue of the Schreier graph on $n$ vertices
depicting the action of $S_{n}$ on $\left\{ 1,\ldots,n\right\} $.
Inspired by this seminal result, we study similar questions for other
types of sets in $S_{n}$. Specifically, we consider normal sets:
sets that are invariant under conjugation. Relying on character bounds
due to Larsen and Shalev \cite{larsen2008characters}, we show that
for large enough $n$, if $\Sigma\subset S_{n}$ is a full conjugacy
class, then the second eigenvalue of $\mathrm{Cay}\left(S_{n},\Sigma\right)$
is roughly identical to the second eigenvalue of the Schreier graph
depicting the action of $S_{n}$ on ordered $4$-tuples of elements
from $\left\{ 1,\ldots,n\right\} $. We further show that this type
of result does not hold when $\Sigma$ is an arbitrary normal set,
but a slightly weaker one does hold. We state a conjecture in the
same spirit regarding an arbitrary symmetric set $\Sigma\subset S_{n}$,
which yields surprisingly strong consequences.
\end{abstract}
\hypersetup{     colorlinks,     linkcolor={black} }\tableofcontents{}\hypersetup{     colorlinks,     linkcolor=[rgb]{0.8,0.12,0.12} }

\section{Introduction\label{sec:Introduction}}

Consider a finite group $G$ and a symmetric subset $\Sigma\subseteq G$,
namely, $g\in\Sigma\Rightarrow g^{-1}\in\Sigma$. The $\left|G\right|\times\left|G\right|$
adjacency matrix $A$ of the Cayley graph $\cay\left(G,\Sigma\right)$
is symmetric and equals 
\[
\sum\nolimits _{g\in\Sigma}\rho_{\mathrm{reg}}\left(g\right),
\]
where $\mathrm{reg}$ is the right regular representation of $G$,
namely, $\rho_{\mathrm{reg}}\left(g\right)$ is the permutation matrix
depicting multiplication from the right by $g$. Recall that the regular
representation of $G$ decomposes as a direct sum of all complex irreducible
representations of $G$ (irreps for short), each appearing with multiplicity
identical to its dimension\footnote{Throughout this paper we use some standard facts from the theory of
group representations and, more specifically, from the theory of representations
of the symmetric groups $S_{n}$. Good references are \cite{fulton1991representation}
for the general theory, \cite{fulton1997applications} for representations
of $S_{n}$, and \cite{diaconis1988group} for applications to probability.}. An appropriate change of basis thus turns $A$ into a block-diagonal
matrix, with $\dim\left(\rho\right)$ blocks of size $\dim\left(\rho\right)\times\dim\left(\rho\right)$
for every irrep $\rho$ of $G$. The value of each of these $\dim\left(\rho\right)$
blocks is $\sum_{g\in\Sigma}\rho\left(g\right)$. This shows that
the multiset of eigenvalues of $A$ can be partitioned into sub-multisets,
each of which is associated with some $\rho\in\widehat{G}$, where
$\widehat{G}$ is the set of (isomorphism types of) irreps of $G$.
For example, the largest eigenvalue of $A$ is $\left|\Sigma\right|$:
this is the eigenvalue corresponding to the constant eigenfunction,
and it is associated with the trivial representation of $G$.

The current work focuses on the symmetric group $S_{n}$. We consider
the eigenvalues associated with the trivial and the sign representations
to be \emph{trivial}, and denote by \marginpar{$\lambda\left(S_{n},\Sigma\right)$}$\lambda\left(S_{n},\Sigma\right)$
the largest non-trivial eigenvalue of a symmetric set $\Sigma\subseteq S_{n}$.
If we think of $\Sigma$ as the element $\sum_{\sigma\in\Sigma}\sigma$
of $\mathbb{R}\left[S_{n}\right]$, we can thus write
\[
\lambda\left(S_{n},\Sigma\right)=\max_{\rho\in\widehat{S_{n}}\setminus\left\{ \mathrm{triv},\mathrm{sgn}\right\} }\lambda_{1}\left(\rho\left(\Sigma\right)\right);
\]
here $\lambda_{1}\left(\rho\left(\Sigma\right)\right)$ marks the
largest eigenvalue of the matrix $\rho\left(\Sigma\right)=\sum_{\sigma\in\Sigma}\rho\left(\sigma\right)$,
which has only real eigenvalues as $\Sigma$ is symmetric. Alternatively,
$\lambda(S_{n},\Sigma)$ is simply the largest eigenvalue of $A$
whose eigenvector is orthogonal both to the constant functions, and
to the indicator function of $A_{n}\leq S_{n}$.

Many properties of a regular graph are related to the value of its
second largest eigenvalue. Primarily, the spectral gap $\lambda_{1}-\lambda_{2}$
is a good measure for the extent to which the graph is ``expanding''
(see, e.g., the surveys \cite{hoory2006expander,lubotzky2012expander}).
Around 1992, David Aldous conjectured the following: whenever $\Sigma\subset S_{n}$
is a set of transpositions, the largest non-trivial eigenvalue $\lambda\left(S_{n},\Sigma\right)$
is equal to the largest eigenvalue associated with the \emph{standard
}irrep,
\[
\mathrm{std}=\left\{ \left(x_{1},\ldots,x_{n}\right)\in\mathbb{C}^{n}\,\middle|\,x_{1}+\ldots+x_{n}=0\right\} ,
\]
which corresponds to the Young diagram $\left(n-1,1\right)$.\footnote{Occasionally, several different irreps give rise to an eigenvalue
which is equal to $\lambda\left(S_{n},\Sigma\right)$. Aldous' conjecture
says that when $\Sigma$ is a set of transpositions, the standard
irrep is always one of these irreps.} This conjecture was proved in 2009 by Caputo, Liggett and Richthammer.
In fact, they proved a stronger version applying to weighted Cayley
graphs as well:
\begin{thm}[Aldous' spectral gap conjecture, \cite{caputo2010proof}]
\label{thm:Aldous}Let $\Sigma\in\mathbb{R}\left[S_{n}\right]$ be
supported on transpositions, with non-negative coefficients. Then
the second eigenvalue of the weighted Cayley graph $\cay\left(S_{n},\Sigma\right)$
is equal to the largest eigenvalue of the standard representation.
Namely,
\[
\lambda_{2}\left(\cay\left(S_{n},\Sigma\right)\right)=\lambda_{1}\left(\mathrm{std}\left(\Sigma\right)\right).
\]
\end{thm}

This theorem can be equivalently stated in terms of Schreier graphs.
Given $\Sigma\subset S_{n}$, denote $\left[n\right]\stackrel{\mathrm{def}}{=}\left\{ 1,\ldots,n\right\} $
and by $\schone$ the Schreier graph depicting the action of $S_{n}$
on $\left[n\right]$ with respect to the subset $\Sigma$. This is
the $\left|\Sigma\right|$-regular graph with $n$ vertices labeled
$1,\ldots,n$, and with an edge $\left(i,j\right)$ for every $\sigma\in\Sigma$
with $\sigma\left(i\right)=j$ (there may be loops and multiple edges).
The $n$ eigenvalues of (the adjacency matrix of) $\schone$ are a
sub-multiset of the $n!$ eigenvalues of $\mathrm{Cay}\left(S_{n},\Sigma\right)$:
this can be seen directly by lifting every eigenfunction on the Schreier
graph to an eigenfunction with the same eigenvalue in the Cayley graph.
Both the Schreier graph and the Cayley graph have trivial eigenvalue
$\lambda_{1}=\left|\Sigma\right|$. It follows that if $\Sigma$ is
symmetric then
\begin{equation}
\lambda_{2}\left(\schone\right)\le\lambda_{2}\left(\cay\left(S_{n},\Sigma\right)\right).\label{eq:trivial inequality}
\end{equation}
Aldous' conjecture states that whenever $\Sigma$ is a set of transpositions,
there is equality in \eqref{eq:trivial inequality}. This is equivalent
to (the non-weighted version of) Theorem \ref{thm:Aldous}, since
the representation $\mathbb{C}^{\left[n\right]}$ decomposes into
the constant functions in $\mathbb{C}^{\left[n\right]}$, which form
the trivial representation, and $\mathrm{std}$.

Aldous' conjecture cannot be naïvely extended to arbitrary symmetric
sets $\Sigma\subset S_{n}$, as illustrated by the following three
examples:
\begin{example}
\label{exa:counter example to RW1}
\begin{enumerate}
\item \label{enu:transitive subgroups}If $\Sigma$ generates a proper transitive
subgroup $H\le S_{n}$, $H\ne A_{n}$, such as $\left\langle \left(1\,2\,\ldots\,n\right)\right\rangle $,
then $\cay\left(S_{n},\Sigma\right)$ has at least three connected
components and so $\lambda\left(S_{n},\Sigma\right)=\left|\Sigma\right|$,
whereas the Schreier graph $\schone$ is connected and so $\lambda_{1}\left(\mathrm{std}\left(\Sigma\right)\right)=\lambda_{2}\left(\schone\right)\lneq\left|\Sigma\right|$.
\item \label{enu:transposition and a cycle}For $\Sigma=\{\mathrm{id},\left(1~2\right),\left(1~2~\ldots~n\right)^{\pm1}\}$,
it follows from \cite{diaconis1993comparison} that $\lambda_{1}(\std\left(\Sigma\right))<\max(\lambda_{1}(\rho_{\left(n-2,2\right)}(\Sigma)),\lambda_{1}(\rho_{\left(n-2,1,1\right)}(\Sigma)))$,
where $\rho_{\lambda}$ is the irrep with Young diagram $\lambda$;
we elaborate in Example \ref{exa:n-cycle and (12)}.
\item \label{enu:Full-conjugacy-class}Full conjugacy classes also occasionally
have $\lambda\left(S_{n},\Sigma\right)$ not associated with $\mathrm{std}$.
For example, denoting by $\sigma^{S_{n}}$ the conjugacy class of
$\sigma$, if $n$ is odd and $\sigma=\left(1~2\right)\left(3~4\right)\ldots\left(n-2~n-1\right)$,
then $\lambda_{1}(\mathrm{std}(\sigma^{S_{n}}))=0$ while $\lambda(S_{n},\sigma^{S_{n}})=\left|\sigma^{S_{n}}\right|/n$.
We explain this computation below.
\end{enumerate}
\end{example}

But what if instead of considering the Schreier graph depicting the
action of $S_{n}$ on $\left[n\right]$, we consider the Schreier
graph depicting the action on $2$-tuples, $4$-tuples or $k$-tuples
of distinct elements from $\left[n\right]$, for some fixed $k$?
We denote this Schreier graph, which is a $\left|\Sigma\right|$-regular
graph on $n\left(n-1\right)\cdots\left(n-k+1\right)$ vertices, by
$\schk k$.

A more sophisticated variation of Aldous' conjecture is then suggested
by the classification of multiply-transitive finite groups. This classification,
which follows from the classification of finite simple groups, gives
a full description of (faithful) $k$-transitive actions of finite
groups, for every $k\ge2$. In particular, a finite $4$-transitive
group is either $S_{n}$ ($n\ge4$), $A_{n}$ $\left(n\ge6\right)$,
or one of the four Mathieu groups $M_{11}$, $M_{12}$, $M_{23}$
and $M_{24}$ (where $M_{n}$ is a subgroup of $S_{n}$ and is $4$-transitive
in its action on $\left[n\right]$, for $n=11,12,23,24$) \cite[Theorem 4.11]{cameron1999permutation}.
It follows that for $n\ge25$, if $\Sigma\subset S_{n}$ does not
generate $A_{n}$ or $S_{n}$, then $\left\langle \Sigma\right\rangle $
is not $4$-transitive, and so the Schreier graph associated with
the action of $S_{n}$ on $4$-tuples of distinct elements in $\left[n\right]$
is not connected. This shows that the mere existence of a (positive)
spectral gap is captured by $\schk 4$, and this resolves the issue
illustrated in Example \ref{exa:counter example to RW1}\eqref{enu:transitive subgroups}.
Is it possible that $\schk 4$ can capture not only the existence
of a spectral gap but also its exact value?

As elaborated in Example \ref{exa:n-cycle and (12)} below, replacing
$\schone$ with $\schk 2$ resolves the issue with Example \ref{exa:counter example to RW1}\eqref{enu:transposition and a cycle}
as well. As for a full conjugacy class, it turns out that occasionally
one needs to consider also the ``signed'' action of $S_{n}$ on
$k$-tuples. We define it more generally, for weighted subsets, keeping
the generality of Theorem \ref{thm:Aldous}:
\begin{defn}
\label{def:lambda-2 of RW-k}For an action of $S_{n}$ on a set $X$
and an element $\Sigma=\sum_{\sigma\in S_{n}}\alpha_{\sigma}\sigma\in\mathbb{R}\left[S_{n}\right]$,
let $\rho_{X}\left(\sigma\right)$ be the permutation matrix depicting
the action of $\sigma$ on $X$, and denote
\[
\rho_{X}\left(\Sigma\right)=\sum\nolimits _{\sigma\in S_{n}}\alpha_{\sigma}\rho_{X}\left(\sigma\right),\qquad\rho_{X}^{\sgn}\left(\Sigma\right)=\left(\rho_{X}\otimes\sgn\right)\left(\Sigma\right)=\sum\nolimits _{\sigma\in S_{n}}\sgn\left(\sigma\right)\alpha_{\sigma}\rho_{X}\left(\sigma\right).
\]
We denote by $\lambda\left(k,\Sigma\right)$ \marginpar{$\lambda\left(k,\Sigma\right)$ $\lambda^{\protect\sgn}\left(k,\Sigma\right)$}(resp.\ $\lambda^{\sgn}\left(k,\Sigma\right)$)
the largest eigenvalue of $\rho_{\left[n\right]_{k}}\left(\Sigma\right)$
(resp.\ $\rho_{\left[n\right]_{k}}^{\sgn}\left(\Sigma\right)$) corresponding
to an eigenvector orthogonal to the constant functions.
\end{defn}

We call $\Sigma$ \emph{symmetric} if $\alpha_{\sigma}=\alpha_{\sigma^{-1}}$
for every $\sigma\in S_{n}$ and \emph{non-negative} if $\alpha_{\sigma}\ge0$
for every $\sigma\in S_{n}$. It follows from the decomposition of
the regular representation that for every symmetric non-negative $\Sigma$
and $n>k$
\[
\max\left(\lambda\left(k,\Sigma\right),\lambda^{\sgn}\left(k,\Sigma\right)\right)\le\lambda\left(S_{n},\Sigma\right),
\]
similarly to \eqref{eq:trivial inequality}. The discussion above
leads to the following potential generalization of Theorem \ref{thm:Aldous},
which was raised during discussions between Gady Kozma and the second
author.

\begin{ques}\label{question:false-conj-sch}Is there a fixed $k\geq1$
such that for every large enough $n$ and every symmetric non-negative
$\Sigma\in\mathbb{R}\left[S_{n}\right]$, we have
\[
\max\left(\lambda\left(k,\Sigma\right),\lambda^{\sgn}\left(k,\Sigma\right)\right)=\lambda\left(S_{n},\Sigma\right)?
\]

\end{ques}

This question has an equivalent representation-theoretic formulation.
When $k$ is fixed, the decomposition of the (signed) action of $S_{n}$
on $k$-tuples into irreps, contains a fixed number of natural families
of irreps. For example, for $n\ge8$, the action on $4$-tuples decomposes
precisely to the $12$ irreps associated with Young diagrams with
at most four blocks outside the first row. These families are formally
defined as follows:
\begin{defn}
\label{def:family of irreps}A series of irreps $\{\rho_{n}\in\hat{S_{n}}\}_{n\ge n_{0}}$
is called a \emph{family of irreps} if one of the following conditions
holds:
\begin{enumerate}
\item \label{enu:row}Either the structure of the associated Young diagram
outside the first row is constant, namely, for every $n\ge n_{0}$,
$\rho_{n+1}$ is obtained from $\rho_{n}$ by adding a block to the
first row, or
\item \label{enu:column}The structure of the associated Young diagram outside
the first column is constant, namely, for every $n\ge n_{0}$, $\rho_{n+1}$
is obtained from $\rho_{n}$ by adding a block to the first column.
\end{enumerate}
\end{defn}

The representation $\mathbb{C}^{[n]_{k}}$ of $S_{n}$ decomposes
into all irreps with at most $k$ blocks outside the first row (with
some multiplicities). Similarly, the \emph{signed} action on $k$-tuples,
$\sgn\otimes\mathbb{C}^{[n]_{k}}$, decomposes into all irreps with
at most $k$ blocks outside the first column. Thus, the following
question is equivalent to Question \ref{question:false-conj-sch}.

\begin{ques}\label{question:false-conj}Is there a finite set of
families of irreps $\rho^{\left(1\right)},\ldots,\rho^{\left(m\right)}$
as in Definition \ref{def:family of irreps} such that for every large
enough $n$ and every symmetric non-negative $\Sigma\in\mathbb{R}\left[S_{n}\right]$,
we have that 
\[
\lambda\left(S_{n},\Sigma\right)=\max_{i=1}^{m}\lambda_{1}\left(\rho_{n}^{\left(i\right)}\left(\Sigma\right)\right)?
\]

\end{ques}

This work studies a special case of this question, where $\Sigma$
is \emph{normal}, in the sense that the coefficients $\alpha_{\sigma}$
are constant on every conjugacy class\footnote{Note that a normal element $\Sigma\in S_{n}$ is symmetric, as $\sigma^{-1}$
is conjugate to $\sigma$ for every $\sigma\in S_{n}$.}. In this case, the eigenvalues of the different irreps can be computed
directly from character values (see Lemma \ref{lem:evalues of normal sets}
below). Our first result gives a positive answer to Question \ref{question:false-conj}
(and to Question \ref{question:false-conj-sch}) when $\Sigma$ is
a single conjugacy class. For every $n\ge8$, consider the following
set of eight irreps:
\[
\T_{n}=\left\{ \begin{array}{cccccc}
\left(n-1,1\right)^{\phantom{t}} & \left(n-2,2\right)^{\phantom{t}} &  & \left(n-3,3\right) & \left(n-3,2,1\right) & \left(n-4,4\right)\\
\left(n-1,1\right)^{t} & \left(n-2,2\right)^{t} & \left(n-2,1,1\right)^{t}
\end{array}\right\} \subset\widehat{S_{n}};
\]
Here $\rho^{t}$ denotes the representation $\sgn\otimes\rho$, whose
Young diagram is given by transposing that of $\rho$, e.g.\ $\left(n-1,1\right)^{t}=\left(2,1,1,\ldots,1\right)$.
\begin{thm}
\label{thm:one conj class}There exists $N_{0}\in\mathbb{N}$ such
that for every $n\ge N_{0}$, if $\Sigma\subset S_{n}$ is a full,
single conjugacy class, then $\lambda\left(S_{n},\Sigma\right)$ is
attained by one of the eight irreps in $\T_{n}$:
\[
\lambda\left(S_{n},\Sigma\right)=\max_{\rho\in\T_{n}}\lambda_{1}\left(\rho\left(\Sigma\right)\right).
\]
\end{thm}

In the language of Definition \ref{def:lambda-2 of RW-k} and Question
\ref{question:false-conj-sch}, this implies that for $n$ large enough,
if $\Sigma\subset S_{n}$ is a conjugacy class, then $\lambda\left(S_{n},\Sigma\right)=\max\left(\lambda\left(4,\Sigma\right),\lambda^{\sgn}\left(2,\Sigma\right)\right)$.
The proof of Theorem \ref{thm:one conj class} relies heavily on asymptotically
sharp character bounds due to Larsen and Shalev \cite{larsen2008characters}
-- see Section \ref{sec:A-Single-Conjugacy}. The statement of Theorem
\ref{thm:one conj class} does not hold for $n=16$: for $\Sigma=\left[\left(1~2~3~4~5\right)\left(6~7~8~9~10\right)\left(11~12~13~14~15\right)\right]^{S_{16}}$,
the largest non-trivial eigenvalue $\lambda\left(S_{16,}\Sigma\right)$
is associated with the irreps $\left(11,5\right)$ and its transpose.
However, simulations suggest that this is the largest counter-example:
\begin{conjecture}
Theorem \ref{thm:one conj class} holds with $N_{0}=17$.
\end{conjecture}

When $\Sigma\in\mathbb{R}\left[S_{n}\right]$ is a general non-negative
normal element, not only are the eight irreps from Theorem \ref{thm:one conj class}
insufficient, but no finite set of families of irreps as in Definition
\ref{def:family of irreps} suffices to capture $\lambda\left(S_{n},\Sigma\right)$,
and so the answer to Questions \ref{question:false-conj-sch} and
\ref{question:false-conj} turns out to be negative:
\begin{thm}
\label{thm:normal-sets-negative-result}For every $k\geq1$ and every
large enough $n$, there is a non-negative normal element $\Sigma\in\mathbb{R}\left[S_{n}\right]$
such that 
\[
\max\left(\lambda\left(k,\Sigma\right),\lambda^{\sgn}\left(k,\Sigma\right)\right)\lneq\lambda\left(\cay\left(S_{n},\Sigma\right)\right).
\]
\end{thm}

Stated differently, no family of irreps as in Definition \ref{def:family of irreps}
suffices to capture $\lambda\left(S_{n},\Sigma\right)$ for general
non-negative normal $\Sigma$. However, our analysis for the case
of a single conjugacy class does readily show the following.
\begin{thm}
\label{thm:normal-sets-positive-result}Let $\Sigma=\sum\alpha_{\sigma}\sigma\in\mathbb{R}[S_{n}]$
be non-negative and normal, and denote $\left|\Sigma\right|\stackrel{\mathrm{def}}{=}\rho_{triv}\left(\Sigma\right)=\sum_{\sigma}\alpha_{\sigma}$.
Then the spectral gap $\left|\Sigma\right|-\lambda\left(S_{n},\Sigma\right)$
is bounded by the spectral gap of the Schreier graph $\schone$ multiplied
by a decaying multiplicative factor:
\[
\left|\Sigma\right|-\lambda\left(S_{n},\Sigma\right)\ge\left[\left|\Sigma\right|-\lambda\left(1,\Sigma\right)\right]\cdot\left[1-o_{n}\left(1\right)\right].
\]
\end{thm}

All the evidence we have so far supports the following generalization
of Aldous' spectral gap conjecture (Theorem \ref{thm:Aldous}), which
was raised, as was Question \ref{question:false-conj-sch} above,
during discussions between Gady Kozma and the second author:
\begin{conjecture}[Kozma-Puder]
\label{conj:kozma-puder} There is a number $k\geq4$ and a universal
constant $0<c<1$, such that for large enough $n$ and for every symmetric
non-negative $\Sigma\in\mathbb{R}\left[S_{n}\right]$,
\[
\left|\Sigma\right|-\lambda\left(S_{n},\Sigma\right)\ge c\cdot\left[\left|\Sigma\right|-\max\left(\lambda\left(k,\Sigma\right),\lambda^{\sgn}\left(k,\Sigma\right)\right)\right].
\]
\end{conjecture}

This conjecture, if true, would have far-reaching consequences. It
would yield that random pairs of permutations in $S_{n}$ give rise
to a uniform family of expanders, which is a long standing open question
(e.g., \cite[Problem 2.28]{lubotzky2012expander}). It would also
yield a conjecture of Babai \cite[Conjecture 1.7]{babai1992diameter}
that for every generating set $\Sigma$ of $A_{n}$, the diameter
of $\cay\left(A_{n},\Sigma\right)$ is bounded by some $n^{c}$ where
$c$ is a universal constant. See Section \ref{sec:A-Conjecture}
for more details.
\begin{rem}
There have been a few attempts to find phenomena as the ones described
here in families of groups other than the symmetric groups. Recently,
Cesi found an analog of Aldous' conjecture in signed symmetric groups
\cite{cesi2020spectral}. Greenhut found a small set of irreps of
the groups $\mathrm{SL}_{n}\left(\mathbb{F}_{q}\right)$ which detect
the mere existence of a spectral gap for all $n$ and prime powers
$q$ \cite{Greenhut2020}.
\end{rem}

The paper is organized as follows. In Section \ref{sec:A-Single-Conjugacy}
we consider sets consisting of a single conjugacy class and prove
Theorem \ref{thm:one conj class}. Section \ref{sec:Arbitrary-Normal-Sets}
deals with arbitrary normal sets and contains the proofs of Theorems
\ref{thm:normal-sets-negative-result} and \ref{thm:normal-sets-positive-result}.
In Section \ref{sec:A-Conjecture} we further discuss Conjecture \ref{conj:kozma-puder}
and its consequences.

\section{A Single Conjugacy Class\label{sec:A-Single-Conjugacy}}

We start with the following standard lemma\footnote{This lemma was popularized by Diaconis, for example in his book \cite{diaconis1988group}.},
which explains why all eigenvalues of $\cay\left(G,\Sigma\right)$
can be read off from the character table of $G$ when $\Sigma\in\mathbb{R}\left[G\right]$
is normal. The important quantity here is the normalized character
of $\rho\in\widehat{S_{n}}$ which we denote by $\x_{\rho}$\marginpar{$\protect\x_{\rho}$}:
\[
\x_{\rho}\left(\sigma\right)\stackrel{\mathrm{def}}{=}{\textstyle \frac{\chi_{\rho}\left(\sigma\right)}{\chi_{\rho}\left(1\right)}}\qquad\left(\sigma\in S_{n}\right),
\]
where $\chi_{\rho}\colon\sigma\mapsto\mathrm{trace}\,(\rho(\sigma))$
is the character of $\rho$. The character table of $S_{n}$ consists
only of integers (cf.\ \cite{fulton1997applications}), and $\left|\chi_{\rho}\left(\sigma\right)\right|\le\chi_{\rho}\left(1\right)$,
so that $\x_{\rho}\left(\sigma\right)\in\mathbb{Q}\cap\left[-1,1\right]$
for every $\rho\in\widehat{S_{n}}$, $\sigma\in S_{n}$.
\begin{lem}
\label{lem:evalues of normal sets}Let $G$ be a finite group, $\Sigma\in\mathbb{R}[G]$
a normal element, and $\rho\in\widehat{G}$. Denoting by $\alpha_{C}$
the coefficient in $\Sigma$ of each $\sigma$ in the conjugacy class
$C$, and by $\x_{\rho}\left(C\right)$ the value of $\tilde{\chi}_{\rho}$
on $C$, the matrix $\rho\left(\Sigma\right)$ equals the scalar 
\begin{equation}
\sum\nolimits _{C\in\mathrm{Conj}\left(G\right)}\alpha_{C}\left|C\right|\x_{\rho}\left(C\right).\label{eq:evalue of normal}
\end{equation}
\end{lem}

\begin{proof}
Since $\Sigma$ is in the center of $\mathbb{C}[G]$, $\rho\left(\Sigma\right)$
is an endomorphism of an irreducible representation, hence a scalar
by Schur's Lemma, and the trace of $\rho\left(\Sigma\right)$ is $\sum_{C}\alpha_{C}\left|C\right|\chi_{\rho}\left(C\right)$.
\end{proof}
This section studies the case of a single conjugacy class $\Sigma=\mathbf{1}_{C}$,
which means that 
\begin{equation}
\lambda\left(S_{n},\Sigma\right)=\left|C\right|\cdot\max_{\rho\in\widehat{S_{n}}\setminus\left\{ \mathrm{triv},\mathrm{sgn}\right\} }\x_{\rho}\left(C\right).\label{eq:evalues for a single conj class}
\end{equation}
In this case Theorem \ref{thm:one conj class} states that for every
$n\ge N_{0}$ and $\sigma\in S_{n}$ we have
\begin{equation}
\max_{\rho\in\widehat{S_{n}}\setminus\left\{ \triv,\sgn\right\} }\x_{\rho}\left(\sigma\right)=\max_{\rho\in\T_{n}}\x_{\rho}\left(\sigma\right).\label{eq:ten irreps rule}
\end{equation}
For a given $\sigma\in S_{n}$, if the maximum in the left hand side
of \eqref{eq:ten irreps rule} is obtained by some $\rho\in\widehat{S_{n}}\setminus\left\{ \triv,\sgn\right\} $,
we say that ``$\rho$ rules\marginpar{$\rho$ rules} for $\sigma$
in $S_{n}$''.

Our main tool in analyzing the normalized characters $\x_{\rho}\left(\sigma\right)$
is the following asymptotically sharp character bounds established
by Larsen and Shalev. We use the notation $c_{\ell}\left(\sigma\right)$\marginpar{$c_{\ell}\left(\sigma\right)$}
for the number of $\ell$-cycles in the permutation $\sigma\in S_{n}$.
For example, $c_{1}\left(\sigma\right)$ is the number of fixed points.
\begin{thm}[{\cite[Theorem 1.3]{larsen2008characters}}]
\label{thm:larsen-shalev}Let $\sigma\in S_{n}$ and let $f=\max\left(c_{1}\left(\sigma\right),1\right)$.
For every irrep $\rho\in\widehat{S_{n}}$, its character $\chi_{\rho}$
satisfies
\begin{equation}
\x_{\rho}\left(\sigma\right)\le\left|\chi_{\rho}\left(1\right)\right|^{-\frac{\log\left(n/f\right)}{2\log n}+\varepsilon_{n}},\label{eq:Larsen shalev}
\end{equation}
where $\varepsilon_{n}$ is a real number tending to $0$ as $n\to\infty$.
\end{thm}

Our strategy in proving Theorem \ref{thm:one conj class} is as follows:
using Theorem \ref{thm:larsen-shalev}, we show that for large enough
$n$, if $\sigma\in S_{n}$ has \textbf{exactly two} fixed points,
then the standard representation $\std$ rules, namely, the maximal
normalized character is $\x_{\std}\left(\sigma\right)=\frac{1}{n-1}$.
Using a simple induction argument we then show that the same is true
for every large enough $n$ when $c_{1}\left(\sigma\right)\ge2$.
Finally, we use Theorem \ref{thm:larsen-shalev} again to deal with
the case $c_{1}\left(\sigma\right)\in\left\{ 0,1\right\} $. Indeed,
Theorem \ref{thm:one conj class} follows immediately from the following
two propositions, which we prove in the following two subsections.
\begin{prop}
\label{prop:c1 ge 2}There is some $N_{1}\in\mathbb{N}$ such that
for every $n\ge N_{1}$ and every $\sigma\in S_{n}$ with $c_{1}\left(\sigma\right)\ge2$,
the standard irrep $\std=\left(n-1,1\right)$ rules.
\end{prop}

\begin{prop}
\label{prop:c1=00003D0 or c1=00003D1}There is some $N_{2}\in\mathbb{N}$
such that for every $n\ge N_{2}$ and every $\sigma\in S_{n}$ with
$c_{1}\left(\sigma\right)\le1$, one of the irreps in $\T_{n}$ rules.
\end{prop}

\subsection{The case $\mathbf{c_{1}\left(\sigma\right)\ge2}$}

The following lemma goes back to Frobenius (cf.\ \cite[§I.7.14]{macdonald1979symmetric}).
We give its proof for completeness.
\begin{lem}
\label{lem:poly in c_1,c_2,..}Let $\left\{ \rho_{n}\in\smash{\widehat{S_{n}}}\right\} _{n\ge n_{0}}$
be a family of irreps as in Definition \ref{def:family of irreps}
with constant structure outside the first row, so that the first row
of $\rho_{n}$ has exactly $n-k$ blocks. Then there is a polynomial
$p\in\mathbb{Q}\left[c_{1},\ldots,c_{k}\right]$ so that for $n\ge2k$
and $\sigma\in S_{n}$, 
\begin{align*}
\chi_{\rho_{n}}\left(\sigma\right) & =p\left(c_{1}\left(\sigma\right),\ldots,c_{k}\left(\sigma\right)\right),\ \text{and}\\
\chi_{\rho_{n}^{t}}\left(\sigma\right) & =\sgn\left(\sigma\right)\cdot p\left(c_{1}\left(\sigma\right),\ldots,c_{k}\left(\sigma\right)\right).
\end{align*}
In particular, $\dim\rho_{n}=\chi_{\rho_{n}}\left(1\right)$ is given
by a polynomial in $n$, equal to $p\left(n,0,\ldots,0\right)$.
\end{lem}

The fact that $\chi_{\rho_{n}}\left(1\right)=dim$$\rho_{n}$ is given
by a polynomial in $n$ is also evident from the hook length formula
(e.g., \cite[Formula 4.12]{fulton1991representation}). The polynomials
associated with every family of irreps with at most four blocks outside
the first row are listed in Table \ref{tab:Dimensions-and-characters}.
\begin{proof}
For a partition $\lambda=\left(\lambda^{\left(1\right)},\ldots,\lambda^{\left(\ell\right)}\right)\vdash n$,
let $M^{\lambda}$ denote the reducible representation associated
with $\lambda$: this is the permutation representation describing
the action of $S_{n}$ on partitions of $\left\{ 1,\ldots,n\right\} $
with sizes of blocks given by $\lambda$ (see \cite[\S 7.2]{fulton1997applications}).
It is not hard to see that the statement of the lemma holds for families
$\left\{ M^{\lambda_{n}}\right\} _{n\ge n_{0}}$ when $\lambda_{n+1}$
is obtained from $\lambda_{n}$ by increasing $\lambda^{\left(1\right)}$
by one. For example, if $\lambda_{n}=\left(n-3,2,1\right)$, then
$\chi_{M^{\lambda_{n}}}\left(\sigma\right)$, which equals the number
of fixed points of $\sigma$ in this action of $S_{n}$, is $c_{1}\cdot\binom{c_{1}-1}{2}+c_{1}c_{2}$.

For every $\lambda\vdash n$ as above, the character $\chi_{\rho^{\lambda}}$
of the irrep $\rho^{\lambda}$ corresponding to $\lambda$ is given
by a linear combination with integer coefficients of the representations
$\left\{ M^{\mu}\,\middle|\,\mu\triangleleft\lambda\right\} $, where
``$\triangleleft$'' marks the dominance relation (see \cite[\S 7.2]{fulton1997applications}).
Moreover, when $\lambda^{\left(1\right)}=n-k\ge\frac{n}{2}$, this
linear combination is independent of $n$, namely, the coefficient
of every $\mu\triangleleft\lambda$ depends only on the structure
of $\mu$ outside the first row. For example, 
\[
\chi_{\rho^{\left(n-2,1,1\right)}}=\chi_{M^{\left(n-2,1,1\right)}}-\chi_{M^{\left(n-2,2\right)}}-\chi_{M^{\left(n-1,1\right)}}+\chi_{M^{\left(n\right)}}.
\]
The statement of the lemma follows.
\end{proof}
\begin{table}
\begin{centering}
\begin{tabular}{|c|c|c|}
\hline 
$\rho$ & $\dim\left(\rho\right)=\chi_{\rho}\left(1\right)$ & $\chi_{\rho}\left(\sigma\right)$ when $c_{i}\left(\sigma\right)=c_{i}$\tabularnewline
\hline 
\hline 
$\left(n\right)$ & $1$ & $1$\tabularnewline
\hline 
$\left(n-1,1\right)$ & $n-1$ & $c_{1}-1$\tabularnewline
\hline 
$\left(n-2,2\right)$ & $\frac{n\left(n-3\right)}{2}$ & $\frac{c_{1}\left(c_{1}-3\right)}{2}+c_{2}$\tabularnewline
\hline 
$\left(n-2,1,1\right)$ & $\frac{\left(n-1\right)\left(n-2\right)}{2}$ & $\frac{\left(c_{1}-1\right)\left(c_{1}-2\right)}{2}-c_{2}$\tabularnewline
\hline 
$\left(n-3,3\right)$ & $\frac{n\left(n-1\right)\left(n-5\right)}{6}$ & $\frac{c_{1}\left(c_{1}-1\right)\left(c_{1}-5\right)}{6}+\left(c_{1}-1\right)c_{2}+c_{3}$\tabularnewline
\hline 
$\left(n-3,2,1\right)$ & $\frac{n\left(n-2\right)\left(n-4\right)}{3}$ & $\frac{c_{1}\left(c_{1}-2\right)\left(c_{1}-4\right)}{3}-c_{3}$\tabularnewline
\hline 
$\left(n-3,1,1,1\right)$ & $\frac{\left(n-1\right)\left(n-2\right)\left(n-3\right)}{6}$ & $\frac{\left(c_{1}-1\right)\left(c_{1}-2\right)\left(c_{1}-3\right)}{6}-\left(c_{1}-1\right)c_{2}+c_{3}$\tabularnewline
\hline 
$\left(n-4,4\right)$ & $\frac{n\left(n-1\right)\left(n-2\right)\left(n-7\right)}{24}$ & $\frac{c_{1}\left(c_{1}-1\right)\left(c_{1}-2\right)\left(c_{1}-7\right)}{24}\!+\!\frac{\left(c_{1}^{2}-3c_{1}-1\right)c_{2}}{2}\!+\!\frac{c_{2}^{2}}{2}\!+\!\left(c_{1}\!-\!1\right)c_{3}\!+\!c_{4}$\tabularnewline
\hline 
$\left(n-4,3,1\right)$ & $\frac{n\left(n-1\right)\left(n-3\right)\left(n-6\right)}{8}$ & $\frac{c_{1}\left(c_{1}-1\right)\left(c_{1}-3\right)\left(c_{1}-6\right)}{8}+\frac{\left(c_{1}^{2}-3c_{1}+3\right)c_{2}}{2}-\frac{c_{2}^{2}}{2}-c_{4}$\tabularnewline
\hline 
$\left(n-4,2,2\right)$ & $\frac{n\left(n-1\right)\left(n-4\right)\left(n-5\right)}{12}$ & $\frac{c_{1}\left(c_{1}-1\right)\left(c_{1}-4\right)\left(c_{1}-5\right)}{12}+\left(c_{2}-2\right)c_{2}-\left(c_{1}-1\right)c_{3}$\tabularnewline
\hline 
$\left(n-4,2,1,1\right)$ & $\frac{n\left(n-2\right)\left(n-3\right)\left(n-5\right)}{8}$ & $\frac{c_{1}\left(c_{1}-2\right)\left(c_{1}-3\right)\left(c_{1}-5\right)}{8}-\frac{\left(c_{1}^{2}-3c_{1}-1\right)c_{2}}{2}-\frac{c_{2}^{2}}{2}+c_{4}$\tabularnewline
\hline 
$\negmedspace\left(n-4,1,1,1,1\right)\negmedspace$ & $\negmedspace\frac{\left(n-1\right)\left(n-2\right)\left(n-3\right)\left(n-4\right)}{24}\negmedspace$ & $\negmedspace\frac{\left(c_{1}-1\right)\left(c_{1}-2\right)\left(c_{1}-3\right)\left(c_{1}-4\right)}{24}\!-\!\frac{\left(c_{1}^{2}-3c_{1}+3\right)c_{2}}{2}\!+\!\frac{c_{2}^{2}}{2}\!+\!(c_{1}\!-\!1)c_{3}-c_{4}\negmedspace$\tabularnewline
\hline 
\end{tabular}
\par\end{centering}
\caption{\label{tab:Dimensions-and-characters}Dimensions and characters of
irreps with at most four blocks outside the first row.}
\end{table}

\begin{lem}
\label{lem:dim > n^2}Let $n\ge13$, and let $\rho\in\widehat{S_{n}}$
be an irrep whose Young diagram has at least three blocks outside
the first row and at least three blocks outside the first column.
Then $\dim\rho\ge n^{2.05}$.
\end{lem}

\begin{proof}
If $\rho$ has exactly three blocks outside the first row then $\rho$
is one of $\left(n-3,3\right)$, $\left(n-3,2,1\right)$ or $\left(n-3,1,1,1\right)$,
in which case its dimension is $\frac{n\left(n-1\right)\left(n-5\right)}{6}$,
$\frac{n\left(n-2\right)\left(n-4\right)}{3}$ or $\frac{\left(n-1\right)\left(n-2\right)\left(n-3\right)}{6}$,
respectively. In each of these cases $\dim\rho\ge n^{2.05}$ for $n\ge13$.
The transpose case where $\rho$ has exactly three blocks outside
the first column is identical.

For the case $\rho$ has at least four blocks outside the first row/column
we use induction on $n$. It is easy to check directly (on a computer)
that the statement is true for $n=13,14$: all 93 irreps of $S_{13}$
and 127 irreps of $S_{14}$ satisfying the assumption of the lemma
have dimension $\ge n^{2.05}$. For $n\ge15$, we assume the statement
holds for $n-1$ and for $n-2$ and that $\rho\in\widehat{S_{n}}$
has at least four blocks outside the first row and at least four blocks
outside the first column. By the branching rule, $\dim\rho=\sum_{\rho'=\rho-\square}\dim\rho'$,
the sum being over all $\rho'\in\widehat{S_{n-1}}$ obtained from
$\rho$ by removing one block. If the Young diagram corresponding
to $\rho$ is not a rectangle, there are at least two such $\rho'$,
each with at least three blocks outside the first row and outside
the first column, and we are done as $2\left(n-1\right)^{2.05}\ge n^{2.05}$
for $n\ge15$. Finally, if $\rho$ is a rectangle, it has at least
$\frac{n}{2}>4$ blocks outside the first row and outside the first
column, and there are exactly two ways to remove two blocks from $\rho$.
The branching rule now gives $\dim\rho=\dim\rho_{1}+\dim\rho_{2}$
with $\rho_{1},\rho_{2}\in\widehat{S_{n-2}}$ satisfying the assumption
in the lemma. We are done as $2\left(n-2\right)^{2.05}\ge n^{2.05}$
for $n\ge15$.
\end{proof}
\begin{lem}
\label{lem:2 fixed points}There is some $N_{3}\in\mathbb{N}$ such
that $\std$ rules for every $n\ge N_{3}$ and $\sigma\in S_{n}$
with $c_{1}\left(\sigma\right)=2$. Namely, $\x_{\rho}\left(\sigma\right)\le\x_{\std}\left(\sigma\right)=\frac{1}{n-1}$
for every $\rho\in\widehat{S_{n}}\setminus\left\{ \mathrm{triv},\mathrm{sgn}\right\} $.
\end{lem}

\begin{proof}
Fix $N_{3}\ge13$ so that for every $n\ge N_{3}$, the term $\varepsilon_{n}$
in \eqref{eq:Larsen shalev} satisfies
\[
2.05\cdot\left(-\frac{1}{2}+\frac{\log2}{2\log n}+\varepsilon_{n}\right)\le-1.
\]
Then, if $n\ge N_{3}$ and $\rho\in\widehat{S_{n}}$ satisfies the
assumptions in Lemma \ref{lem:dim > n^2}, it follows from Theorem
\ref{thm:larsen-shalev} that 
\[
\x_{\rho}\left(\sigma\right)\le\left(n^{2.05}\right)^{-\frac{1}{2}+\frac{\log2}{2\log n}+\varepsilon_{n}}\le\frac{1}{n}\le\frac{1}{n-1}.
\]
Finally, if $\rho$ is one of the five remaining irreps $\left(n-1,1\right)^{t}$,
$\left(n-2,2\right)$, $\left(n-2,2\right)^{t}$, $\left(n-2,1,1\right)$
and $\left(n-2,1,1\right)^{t}$, we use the explicit expressions in
Table \ref{tab:Dimensions-and-characters} to show that $\x_{\rho}\left(\sigma\right)\le\frac{1}{n-1}$.
For example, if $\rho=\left(n-2,1,1\right)^{t}$ and $\sigma$ is
odd, we get 
\[
\x_{\rho}\left(\sigma\right)=\frac{2c_{2}\left(\sigma\right)}{\left(n-1\right)\left(n-2\right)}\le\frac{2\cdot\frac{n-2}{2}}{\left(n-1\right)\left(n-2\right)}=\frac{1}{n-1},
\]
so at worst there is a tie between $\left(n-1,1\right)$ and $\left(n-2,1,1\right)^{t}$
when $c_{1}\left(\sigma\right)=2$.
\end{proof}
The following easy but crucial lemma says that the normalized character
of a permutation with at least one fixed point (which we can take
without loss of generality to be $n$) is bounded from both sides
by normalized characters with one block omitted.
\begin{lem}
\label{lem:weighted avg}Let $\sigma\in S_{n}$ be a permutation satisfying
$\sigma\left(n\right)=n$. Then $\sigma$ can be considered also as
an element of $S_{n-1}$, and for every $\rho\in\widehat{S_{n}}$,
the normalized character $\x_{\rho}\left(\sigma\right)$ is a weighted
average of the normalized characters $\left\{ \x_{\rho'}\left(\sigma\right)\right\} _{\rho'=\rho-\square\in\widehat{S_{n-1}}}$.
\end{lem}

Here $\rho'$ runs over all Young diagrams obtained from $\rho$ by
deletion of a single block.
\begin{proof}
By the branching rule $\chi_{\rho}\left(1\right)=\sum_{\rho'=\rho-\square}\chi_{\rho'}\left(1\right)$,
and, similarly $\chi_{\rho}\left(\sigma\right)=\sum_{\rho'=\rho-\square}\chi_{\rho'}\left(\sigma\right)$.
For any real numbers $x_{1},\ldots,x_{k}$ and positive real numbers
$y_{1},\ldots,y_{k}$, $\frac{x_{1}+\ldots+x_{k}}{y_{1}+\ldots+y_{k}}$
is a convex combination of $\frac{x_{1}}{y_{1}},\ldots,\frac{x_{k}}{y_{k}}$.
\end{proof}
\begin{cor}
\label{cor:std rules in n-1 =00003D=00003D> rules in n}For $\sigma\in S_{n}$
with $\sigma(n)=n$, if $\std$ rules for $\sigma$ in $S_{n-1}$,
it also rules for $\sigma$ in $S_{n}$.
\end{cor}

\begin{proof}
Denote the standard irrep in $S_{n-1}$ by $\std'$. Then 
\[
\x_{\std}\left(\sigma\right)=\frac{c_{1}\left(\sigma\right)-1}{n-1}\ge\frac{c_{1}\left(\sigma\right)-2}{n-2}=\x_{\std'}\left(\sigma\right).
\]
On the other hand, if $\rho\in\widehat{S_{n}}\setminus\left\{ \triv,\sgn,\std,\std^{t}\right\} $,
then removing a block from $\rho$ does not yield neither the trivial
nor the sign irreps of $S_{n-1}$, and so by Lemma \ref{lem:weighted avg},
\[
\x_{\rho}\left(\sigma\right)\le\x_{\std'}\left(\sigma\right).
\]
Finally, regarding $\std^{t}$, if $\sigma$ is even then $\x_{\std^{t}}\left(\sigma\right)=\x_{\std}\left(\sigma\right)$,
and if $\sigma$ is odd then as $\std$ rules for $\sigma$ in $S_{n-1}$,
we have $c_{1}\left(\sigma\right)\ge2$ and $\x_{\std^{t}}\left(\sigma\right)=-\x_{\std}\left(\sigma\right)<\x_{\std}\left(\sigma\right)$.
\end{proof}
Recall that the \emph{support}\marginpar{$\protect\supp$} of $\sigma\in S_{n}$
is $\supp\left(\sigma\right)=\left\{ i\in\left\{ 1,\ldots,n\right\} \,\middle|\,\sigma\left(i\right)\ne i\right\} .$
\begin{cor}
\label{cor:c1>=00003D2 and large support}Let $N_{3}$ be as in Lemma
\ref{lem:2 fixed points}. Then $\std$ rules for every $n\ge N_{3}$
and $\sigma\in S_{n}$ with $c_{1}\left(\sigma\right)\ge2$ and $\left|\supp\left(\sigma\right)\right|\ge N_{3}-2$.
\end{cor}

\begin{proof}
Let $k=\left|\supp\left(\sigma\right)\right|+2\ge N_{3}$. By omitting
$n-k$ fixed points from $\sigma$, we may think of $\sigma$ as representing
a conjugacy class in $S_{k}$ with exactly two fixed points. By Lemma
\ref{lem:2 fixed points}, $\std$ rules for $\sigma$ in $S_{k}$.
By applying Corollary \ref{cor:std rules in n-1 =00003D=00003D> rules in n}
$n-k$ times, we deduce that $\std$ also rules for $\sigma$ in $S_{n}$.
\end{proof}
\begin{lem}
\label{lem:std rules eventually for every sigma}Let $\sigma\in S_{r}$
and consider it as a permutation in $S_{n}$ for any $n\ge r$ by
appending $n-r$ fixed points. For large enough $n$, $\std$ rules
for $\sigma$ in $S_{n}$.
\end{lem}

\begin{proof}
By enlarging $r$ a bit first, we can assume without loss of generality
that $r\ge5$, so every $\rho\in\widehat{S_{r}}\setminus\left\{ \triv,\sgn\right\} $
is faithful, and also that $c_{1}\left(\sigma\right)\ge1$ in $S_{r}$,
so if $\std^{t}$ rules for some $n\ge r$, so does $\std$. Assume
that some $\rho\in\widehat{S_{r}}\setminus\left\{ \triv,\sgn\right\} $
rules for $\sigma$ in $S_{r}$. As $\rho$ is faithful, $\x_{\rho}\left(\sigma\right)<1$.
As $n$ increases, as long as $\std$ does not rule, all normalized
characters of irreps in $\widehat{S_{n}}\setminus\left\{ \triv,\sgn\right\} $
are bounded from above by $\x_{\rho}\left(\sigma\right)$: indeed,
by induction this is true for all such irreps in $\widehat{S_{n-1}}$,
and by Lemma \ref{lem:weighted avg} this is also true for every irrep
in $\widehat{S_{n}}\setminus\left\{ \triv,\sgn,\std,\std^{t}\right\} $.
For $\std$ and $\std^{t}$ this is true by the assumption that $\std$
does not rule. In contrast, the normalized character of $\std$, which
is $\frac{n-\left|\supp\left(\sigma\right)\right|-1}{n-1}$, tends
to $1$ as $n\to\infty$. Therefore, for some $n_{0}$, $\std$ rules.
Corollary \ref{cor:std rules in n-1 =00003D=00003D> rules in n} then
implies that $\std$ rules for $\sigma$ for every $n\ge n_{0}$.
\end{proof}
\begin{cor}
\label{cor:std rules for all sigma with small support}Let $M\in\mathbb{N}$
be a constant. There is some $N_{4}=N_{4}\left(M\right)\in\mathbb{N}$
such that for every $n\ge N_{4}$ and every $\sigma\in S_{n}$ with
$\left|\supp\left(\sigma\right)\right|\le M$, $\std$ rules for $\sigma$
in $S_{n}$.
\end{cor}

\begin{proof}
Every such $\sigma$ belongs to a conjugacy class in $S_{n}$ which
is obtained from blowing up (by appending fixed point) a fixed-point-free
conjugacy class in $S_{k}$ for some $k\le M$. This is a finite set
of starting points, which means we need to apply Lemma \ref{lem:std rules eventually for every sigma}
finitely many times.
\end{proof}
\begin{proof}[Proof of Proposition \ref{prop:c1 ge 2}]
Let $N_{3}$ be the constant from Lemma \ref{lem:2 fixed points}
and Corollary \ref{cor:c1>=00003D2 and large support}, and $N_{4}=N_{4}\left(M\right)$
with $M=N_{3}-3$ be the constant from Corollary \ref{cor:std rules for all sigma with small support}.
Set $N_{1}=\max\left(N_{3},N_{4}\right)$. Then $\std$ rules for
every $n\ge N_{1}$ and for every $\sigma\in S_{n}$ with $c_{1}\left(\sigma\right)\ge2$.
\end{proof}

\subsection{The case $\mathbf{c_{1}\left(\sigma\right)\le1}$}

The proof strategy of Proposition \ref{prop:c1=00003D0 or c1=00003D1}
is the same as in the case $c_{1}\left(\sigma\right)=2$, albeit significantly
more tedious. The difference is that the largest normalized character
is not necessarily of order $\frac{1}{n}$ as in the $c_{1}\left(\sigma\right)=2$
case, but can be of order as low as $\frac{1}{n^{3}}$. For example,
this is the case when $\sigma$ is even, $c_{1}\left(\sigma\right)=c_{3}\left(\sigma\right)=1$
and $c_{2}\left(\sigma\right)=0$ (and see Table \ref{tab:ruling irrep for even perms}).

\begin{lem}
\label{lem:monomials in the polynomial for characters}Let $\left\{ \rho_{n}\right\} _{n\ge n_{0}}$
be a family of irreps $\rho_{n}\in\widehat{S_{n}}$ as in Definition
\ref{def:family of irreps} with $k$ blocks outside the first row
or outside the first column. Every monomial $c_{1}^{\alpha_{1}}c_{2}^{\alpha_{2}}\cdots c_{k}^{\alpha_{k}}$
in the associated polynomial $p$ from Lemma \ref{lem:poly in c_1,c_2,..}
satisfies $\sum_{i=1}^{k}i\cdot\alpha_{i}\le k$. In addition, the
polynomial $p\left(c_{1},0,\ldots,0\right)$ giving the dimension
of $\rho_{n}$ is of degree exactly $k$.
\end{lem}

\begin{proof}
The first statement holds because it holds for the polynomials depicting
the characters of the reducible representations $M^{\lambda}$ (see
the proof of Lemma \ref{lem:poly in c_1,c_2,..}), and the character
of $\rho_{n}$ is equal to a linear combination of $M^{\lambda}$'s
with at most $k$ blocks outside the first row (namely, $\lambda^{\left(1\right)}\ge n-k$).
The second statement is immediate from the hook length formula.
\end{proof}
\begin{lem}
\label{lem:ten at least 2/n^3}For every large enough $n$ and every
$\sigma\in S_{n}$
\[
f\left(\sigma\right)\stackrel{\mathrm{def}}{=}\max_{\rho\in\T_{n}}\x_{\rho}\left(\sigma\right)\ge\frac{3}{n\left(n-2\right)\left(n-4\right)}.
\]
\end{lem}

\begin{proof}
We deal with six different cases in the following table (the values
of normalized characters can be read from Table \ref{tab:Dimensions-and-characters}):\\
\hspace*{\fill}%
\begin{tabular}{|c|c|}
\hline 
Assumptions on $\sigma$ & lower bound on $f\left(\sigma\right)$\tabularnewline
\hline 
\hline 
$c_{1}\ge2$ & $f\left(\sigma\right)\ge\x_{\std}\left(\sigma\right)\ge\frac{1}{n-1}$\tabularnewline
\hline 
$c_{1}=1$, $c_{3}\ge1$ & $f\left(\sigma\right)\ge\x_{\left(n-3,3\right)}\left(\sigma\right)=\frac{6c_{3}}{n\left(n-1\right)\left(n-5\right)}\ge\frac{6}{n\left(n-1\right)\left(n-5\right)}$\tabularnewline
\hline 
$c_{1}=1$, $c_{3}=0$ & $f\left(\sigma\right)\ge\x_{\left(n-3,2,1\right)}\left(\sigma\right)=\frac{3}{n\left(n-2\right)\left(n-4\right)}$\tabularnewline
\hline 
$c_{1}=0$, $c_{2}\ge1$ & $f\left(\sigma\right)\ge\x_{\left(n-2,2\right)}\left(\sigma\right)=\frac{2c_{2}\left(\sigma\right)}{n\left(n-3\right)}\ge\frac{2}{n\left(n-3\right)}$\tabularnewline
\hline 
$c_{1}=0$, $c_{2}=0$, $\sigma$ is even & $f\left(\sigma\right)\ge\x_{\left(n-2,1,1\right)^{t}}\left(\sigma\right)=\frac{2}{\left(n-1\right)\left(n-2\right)}$\tabularnewline
\hline 
$c_{1}=0$, $c_{2}=0$, $\sigma$ is odd & $f\left(\sigma\right)\ge\x_{\std^{t}}\left(\sigma\right)=\frac{1}{\left(n-1\right)}$\tabularnewline
\hline 
\end{tabular}\hspace*{\fill}\\
\end{proof}
\begin{lem}
\label{lem:ten beats any family with k>4}Let $\left\{ \rho_{n}\right\} _{n\ge n_{0}}$
be a family of irreps $\rho_{n}\in\widehat{S_{n}}$ as in Definition
\ref{def:family of irreps} with $k\ge5$ blocks outside the first
row or outside the first column. Then for large enough $n$ and every
$\sigma\in S_{n}$ with $c_{1}\left(\sigma\right)\le1$,
\[
\x_{\rho_{n}}\left(\sigma\right)\le f\left(\sigma\right)\stackrel{\mathrm{def}}{=}\max_{\rho\in\T_{n}}\x_{\rho}\left(\sigma\right).
\]
\end{lem}

\begin{proof}
If $c_{2}\left(\sigma\right)\ge2$ then 
\[
f\left(\sigma\right)\ge\x_{\left(n-2,2\right)}\left(\sigma\right)\ge\frac{2\left(c_{2}\left(\sigma\right)-1\right)}{n\left(n-3\right)}\ge\frac{2}{n\left(n-3\right)}.
\]
In contrast, by Lemma \ref{lem:monomials in the polynomial for characters},
as $c_{2}\left(\sigma\right),\ldots,c_{k}\left(\sigma\right)\le n$,
we have $\chi_{\rho_{n}}\left(\sigma\right)\le O(n^{\left\lfloor k/2\right\rfloor })$
and $\chi_{\rho_{n}}\left(1\right)$ is a polynomial in $n$ of degree
$k$. We deduce that the normalized character satisfies $\x_{\rho_{n}}\left(\sigma\right)\le O\left(n^{\left\lfloor k/2\right\rfloor -k}\right)=O\left(n^{-3}\right)$
as $k\ge5$.

If $c_{2}\left(\sigma\right)\le1$ then by Lemma \ref{lem:ten at least 2/n^3}
$f\left(\sigma\right)\ge\frac{3}{n\left(n-2\right)\left(n-4\right)}$
whereas, by Lemma \ref{lem:monomials in the polynomial for characters},
$\chi_{\rho_{n}}\left(\sigma\right)\le O(n^{\left\lfloor k/3\right\rfloor })$
and so $\x_{\rho_{n}}(\sigma)\le O\left(n^{\left\lfloor k/3\right\rfloor -k}\right)=O(n^{-4})$
as $k\ge5$.
\end{proof}
Of course, Lemma \ref{lem:ten beats any family with k>4} suffices
to deal with every family of irreps separately, but not with all irreps
uniformly. For this, we need to use Larsen-Shalev's Theorem \ref{thm:larsen-shalev}.
First, as above, we need a uniform lower bound on the dimension of
almost all irreps:
\begin{lem}
\label{lem:dim > n^6}Let $n\ge39$, and let $\rho\in\widehat{S_{n}}$
be an irrep represented by a Young diagram with at least 14 blocks
outside the first row and at least 14 blocks outside the first column.
Then $\dim\rho\ge n^{6.05}$.
\end{lem}

\begin{proof}
We verified the statement numerically for $n=39,\ldots,48$. For the
general case, assume that $\rho$ has exactly $14$ blocks outside
the first row (the transpose case is identical). Consider the hook
lengths at the first row of the Young diagram associated with $\rho$.
The hook length at the $15$th block is $n-28$, since the second
row is of length $\le14$. The hook length at the $16th$ block is
$n-29$ and so on. For $1\le i\le14$, the first $\left(i-1\right)$
columns contain at least $2\left(i-1\right)$ blocks, so the hook
length of the $i$th block is at most $n-2\left(i-1\right)$. The
product of hook lengths of all blocks outside the first row is at
most $14!$ (by the hook length formula for Young diagrams with $14$
blocks). Thus, by the hook length formula, we obtain that 
\begin{eqnarray*}
\dim\rho & = & \frac{n!}{\mathrm{product~of~hook~lengths}}\\
 & \ge & \frac{n!}{n\left(n-2\right)\left(n-4\right)\cdots\left(n-26\right)\cdot\left(n-28\right)!\cdot14!}=\frac{\left(n-1\right)\left(n-3\right)\left(n-5\right)\cdots\left(n-27\right)}{14!},
\end{eqnarray*}
which is greater than $n^{6.05}$ for $n\ge47$. For $n\ge49$ and
$\rho\in\widehat{S_{n}}$ with at least $15$ blocks outside the first
row or outside the first column, we proceed by induction exactly as
in the proof of Lemma \ref{lem:dim > n^2}.
\end{proof}

\begin{proof}[Proof of Proposition \ref{prop:c1=00003D0 or c1=00003D1}]
As in the proof of Lemma \ref{lem:2 fixed points}, we deduce from
Theorem \ref{thm:larsen-shalev} and Lemma \ref{lem:dim > n^6} that
for large enough $n$, for all $\sigma\in S_{n}$ with $c_{1}\left(\sigma\right)\le1$
and for all $\rho\in\widehat{S_{n}}$ with at least $14$ blocks outside
the first row or outside the first column, we have 
\[
\x_{\rho}\left(\sigma\right)\le\frac{1}{n^{3}},
\]
which, by Lemma \ref{lem:ten at least 2/n^3}, is less than the maximal
normalized character of $\sigma$ among the irreps in $\T_{n}$. For
the finite number of families of irreps with $5\le k\le13$, where
$k$ is the number of blocks outside the first row or column, we use
Lemma \ref{lem:ten beats any family with k>4}. Finally, the normalized
characters of all irreps with $k\le4$ appear in Table \ref{tab:Dimensions-and-characters}.
Comparing them for large $n$ yields that for every $\sigma\in S_{n}$
with $c_{1}\left(\sigma\right)\le1$, some $\rho\in\T_{n}$ rules.
We omit the technical details of this comparison, but see Tables \ref{tab:ruling irrep for even perms}
and \ref{tab:ruling irrep for odd perms} for the ruling irrep in
each case.
\end{proof}
\begin{table}[h]
\begin{centering}
\begin{tabular}{|>{\centering}m{3cm}|>{\centering}m{7cm}|}
\hline 
$\rho$ & $\rho$ rules for even $\sigma$ when:\tabularnewline
\hline 
\hline 
$\left(n-1,1\right)$ $\left(n-1,1\right)^{t}$ & $c_{1}\ge2$\tabularnewline
\hline 
$\left(n-2,2\right)$

$\left(n-2,2\right)^{t}$ & $c_{1}=1,c_{2}\ge2$

$c_{1}=0,c_{2}\ge1$\tabularnewline
\hline 
$\left(n-2,1,1\right)$

$\left(n-2,1,1\right)^{t}$ & $c_{1}=0,c_{2}=0,c_{3}\le\frac{n-4}{3}$\tabularnewline
\hline 
$\left(n-3,3\right)$

$\left(n-3,3\right)^{t}$ & $c_{1}=1,c_{2}=1,c_{3}\ge1$

$c_{1}=1,c_{2}=0,c_{3}\ge2$

$c_{1}=1,c_{2}=0,c_{3}=1,c_{4}\le\frac{n-5}{4}$

$c_{1}=0,c_{2}=0,c_{3}=\frac{n}{3}$\tabularnewline
\hline 
$\left(n-3,2,1\right)$

$\left(n-3,2,1\right)^{t}$ & $c_{1}=1,c_{2}=1,c_{3}=0,c_{4}\le\frac{n+3}{8}$

$c_{1}=1,c_{2}=0,c_{3}=0,c_{4}\le\frac{n-5}{8}$\tabularnewline
\hline 
$\left(n-4,4\right)$

$\left(n-4,4\right)^{t}$ & $c_{1}=1,c_{2}=1,c_{3}=0,c_{4}\ge\frac{n+4}{8}$

$c_{1}=1,c_{2}=0,c_{3}=1,c_{4}=\frac{n-4}{4}$

$c_{1}=1,c_{2}=0,c_{3}=0,c_{4}\ge\frac{n-4}{8}$\tabularnewline
\hline 
\end{tabular}
\par\end{centering}
\caption{\label{tab:ruling irrep for even perms}For large enough $n$, this
table shows which $\rho\in\protect\T_{n}$ rules for every \textbf{even}
$\sigma\in S_{n}$. Note that at least one of every pair of irreps
in the left column belongs to $\protect\T_{n}$.}
\end{table}
\begin{table}[h]
\begin{centering}
\begin{tabular}{|>{\centering}m{3cm}|>{\centering}m{7cm}|}
\hline 
$\rho$ & $\rho$ rules for odd $\sigma$ when:\tabularnewline
\hline 
\hline 
$\left(n-1,1\right)$ & $c_{1}\ge2$\tabularnewline
\hline 
$\left(n-1,1\right)^{t}$ & $c_{1}=0,c_{2}\le\frac{n-3}{2}$\tabularnewline
\hline 
$\left(n-2,2\right)$ & $c_{1}=0,c_{2}=\frac{n}{2}$\tabularnewline
\hline 
$\left(n-2,2\right)^{t}$ & $c_{1}=1,c_{2}=0$\tabularnewline
\hline 
$\left(n-2,1,1\right)^{t}$ & $c_{1}=2,c_{2}=\frac{n-2}{2}$

$c_{1}=1,c_{2}\ge2$

$c_{1}=1,c_{2}=1,c_{3}\le\frac{n-4}{3}$\tabularnewline
\hline 
$\left(n-3,3\right)$ & $c_{1}=1,c_{2}=1,c_{3}=\frac{n-3}{3}$\tabularnewline
\hline 
\end{tabular}
\par\end{centering}
\caption{\label{tab:ruling irrep for odd perms}For large enough $n$, this
table shows which $\rho\in\protect\T_{n}$ rules for every \textbf{odd}
$\sigma\in S_{n}$.}
\end{table}

\section{Arbitrary Normal Sets\label{sec:Arbitrary-Normal-Sets}}

In this section we prove the negative result, Theorem \ref{thm:normal-sets-negative-result},
stating that no finite set of families of irreps is enough to capture
$\lambda\left(S_{n},\Sigma\right)$ for an arbitrary symmetric non-negative
element $\Sigma\in\mathbb{R}\left[S_{n}\right]$, nor even for a normal
non-negative element $\Sigma\in\mathbb{R}\left[S_{n}\right]$. We
also prove the positive result, Theorem \ref{thm:normal-sets-positive-result},
which says that for every normal non-negative $\Sigma\in\mathbb{R}\left[S_{n}\right]$,
$\std$ alone captures the spectral gap $\left|\Sigma\right|-\lambda\left(S_{n},\Sigma\right)$
up to a decaying multiplicative factor.

\subsection{A negative result}

Before proving Theorem \ref{thm:normal-sets-negative-result}, we
remark that the irreducible characters of $S_{n}$ constitute a linear
basis for the space of class functions on $S_{n}$, hence there are
normal elements $\Sigma\in\mathbb{R}\left[S_{n}\right]$ giving any
prescribed values for the eigenvalues of every irreps as in \eqref{eq:evalue of normal}.
One could still hope that as we restrict to \emph{non-negative} normal
elements, the answer to Question \ref{question:false-conj} would
still be affirmative in the case of normal sets. However, this is
not the case, as we now prove.
\begin{proof}[Proof of Theorem \ref{thm:normal-sets-negative-result}]
Let $\rho^{\left(1\right)},\ldots,\rho^{\left(m\right)}$ be arbitrary
families of non-trivial irreps of $S_{n}$ as in Definition \ref{def:family of irreps}.
Assume each of these families has at most $k$ blocks outside the
first row/column. In particular, for $n\ge2k$, the evaluation of
the characters $\rho^{\left(1\right)},\ldots,\rho^{\left(m\right)}$
on $\sigma\in S_{n}$ depends only on the numbers $c_{1}\left(\sigma\right),\ldots,c_{k}\left(\sigma\right)$
of short cycles -- see Lemma \ref{lem:poly in c_1,c_2,..}. Moreover,
as the characters are given by polynomials in $c_{1},\ldots,c_{k}$,
their expected values for some $\Sigma\in\mathbb{R}\left[S_{n}\right]$
depend only on the \emph{joint distribution} of $\left\{ \left(c_{1}\left(\sigma\right),\ldots,c_{k}\left(\sigma\right)\right)\right\} _{\sigma\in\Sigma}$,
and even more particularly on the distribution in $\Sigma$ of the
monomials $c_{1}^{\alpha_{1}}\cdots c_{k}^{\alpha_{k}}$ with $\sum i\cdot\alpha_{i}\le k$
(see Lemma \ref{lem:monomials in the polynomial for characters}).

Now consider the uniform distribution on $S_{2k}$. By orthogonality
of irreducible characters, for every $\triv\ne\rho\in\widehat{S_{2k}}$,
the expected value of $\chi_{\rho}\left(\sigma\right)$ is zero. Now,
for every $n\ge2k$, if $\Sigma\in\mathbb{R}\left[S_{n}\right]$ is
a normal non-negative element with the same joint distribution of
$c_{1},\ldots,c_{k}$ as the uniform distribution in $S_{2k}$, we
get that the eigenvalues of $\mathrm{Cay}\left(S_{n},\Sigma\right)$
associated with $\rho_{~n}^{\left(1\right)},\ldots,\rho_{~n}^{\left(m\right)}$
all vanish. For large enough $n$, one can construct such $\Sigma$
with large values of, say, $c_{k+1}$, and supported on odd conjugacy
classes. This would assure that there is some irrep $\rho$ with $k+1$
blocks outside the first row or outside the first column with large
positive average value of the character $\chi_{\rho}$. Hence $\rho^{\left(1\right)},\ldots,\rho^{\left(m\right)}$
do not rule for such $\Sigma$.
\end{proof}
\begin{example}
Let us illustrate the proof for the families $\rho^{\left(1\right)}=\left(n-1,1\right)$,
$\rho^{\left(2\right)}=\left(n-2,2\right)$ and $\rho^{\left(3\right)}=\left(n-2,1,1\right)^{t}$
and $n=100$. So we can take $k=2$ and consider the joint distribution
of $\left(c_{1},c_{2}\right)$ in $S_{4}$:\\
\hspace*{\fill}%
\begin{tabular}{|c|c|c|}
\hline 
$c_{1}$ & $c_{2}$ & probability\tabularnewline
\hline 
\hline 
4 & 0 & $1/24$\tabularnewline
\hline 
2 & 1 & $1/4$\tabularnewline
\hline 
1 & 0 & $1/3$\tabularnewline
\hline 
0 & 2 & $1/8$\tabularnewline
\hline 
0 & 0 & $1/4$\tabularnewline
\hline 
\end{tabular} \hspace*{\fill}\medskip{}
\\
Consider the following permutations in $S_{100}$: 
\begin{eqnarray*}
\sigma_{1} & = & \left(1~2~3\right)\left(4~5~6\right)\ldots\left(88~89~90\right)\left(91\ldots96\right)\\
\sigma_{2} & = & \left(1~2~3\right)\left(4~5~6\right)\ldots\left(94~95~96\right)\left(97~98\right)\\
\sigma_{3} & = & \left(1~2~3\right)\left(4~5~6\right)\ldots\left(91~92~93\right)\left(94\ldots99\right)\\
\sigma_{4} & = & \left(1~2~3\right)\left(4~5~6\right)\ldots\left(88~89~90\right)\left(91\ldots96\right)\left(97~98\right)\left(99~100\right)\\
\sigma_{5} & = & \left(1~2~3\right)\left(4~5~6\right)\ldots\left(94~95~96\right)\left(97\ldots100\right)
\end{eqnarray*}
and notice that they all have many $3$-cycles and are all odd. Define
$\Sigma\in\mathbb{R}\left[S_{100}\right]$ by
\[
\Sigma=\sum\nolimits _{i=1}^{5}\alpha_{i}\cdot\sigma_{i}^{S_{100}}
\]
so that $\alpha_{i}\cdot\left|\sigma_{\smash{i}}^{S_{100}}\right|$
is equal to the probability in the $i$th line in the table. For example,
$\alpha_{1}\cdot\left|\sigma_{\smash{1}}^{S_{100}}\right|=\frac{1}{24}$.
This choice of $\Sigma$ assures that for $i=1,\ldots,m$, $\rho^{\left(i\right)}\left(\Sigma\right)$
is the zero matrix, but there is an irrep with three blocks outside
the first row or column with $\rho\left(\Sigma\right)$ being a positive
scalar matrix (in our case, this is true for the irreps $\left(97,3\right)$,
$\left(97,2,1\right)^{t}$ and $\left(97,1,1,1\right)$).
\end{example}

\begin{rem}
\label{rem:n=00003Dk}In fact, it seems that the proof of Theorem
\ref{thm:normal-sets-negative-result} can work also with the joint
distribution of $\left(c_{1},\ldots,c_{k}\right)$ induced from the
uniform distribution on $S_{k}$, rather than on $S_{2k}$. For example,
when $k=2$, the joint distribution in $S_{2}$ of $\left(c_{1},c_{2}\right)$
is $\left(2,0\right)$ and $\left(0,1\right)$ each with probability
$\frac{1}{2}$. For any normal $\Sigma\in\mathbb{R}\left[S_{n}\right]$
with such joint distribution on $\left(c_{1},c_{2}\right)$, the matrix
$\rho\left(\Sigma\right)$ is zero for $\rho\ne\triv,\sgn$ with at
most $2$ blocks outside the first row or the first column.
\end{rem}

\begin{rem}
Theorem \ref{thm:normal-sets-negative-result} is true also if one
restricts attention to normal sets in $S_{n}$, namely, to normal
elements $\Sigma\in\mathbb{R}\left[S_{n}\right]$ with $0-1$ coefficients.
The point is that there is a lot of flexibility in construction $\Sigma$
in the proof, so as $n\to\infty$ one can construct normal sets with
joint distribution of $\left(c_{1},\ldots,c_{k}\right)$ tending to
the one in $S_{2k}$ (or in $S_{k}$), while keeping the average value
of $c_{k+1}$ at least, say, $\frac{n}{2\left(k+1\right)}$, and all
conjugacy classes odd. This assures that some irrep with $k+1$ blocks
outside the first row/column eventually beats every irrep with at
most $k$ blocks outside the first row/column.
\end{rem}

\subsection{A positive result}

We now prove Theorem \ref{thm:normal-sets-positive-result}, based
on the following lemma:
\begin{lem}
\label{lem:std almost rules for any single conj class}Let $N_{0}$
be the constant from Theorem \ref{thm:one conj class}. For every
$n\ge N_{0}$ there is a constant $\delta_{n}>0$ tending to zero
as $n\to\infty$, so that for every $\sigma\in S_{n}$ and $\rho\in\widehat{S_{n}}\setminus\left\{ \triv,\sgn\right\} $
we have 
\begin{equation}
1-\x_{\rho}\left(\sigma\right)\ge\left[1-\x_{\std}\left(\sigma\right)\right]\cdot\left(1-\delta_{n}\right).\label{eq:std almost rules of single conj class}
\end{equation}
\end{lem}

\begin{proof}
By Proposition \ref{prop:c1 ge 2} (and assuming $N_{0}\ge N_{1}$),
if $c_{1}\left(\sigma\right)\ge2$ then $\std$ rules and \eqref{eq:std almost rules of single conj class}
holds even with $\delta_{n}=0$. If $c_{1}\left(\sigma\right)\le1$
then one of $\rho\in\T_{n}$ rules, so it is enough to check \eqref{eq:std almost rules of single conj class}
for each of the seven irreps in $\T_{n}\setminus\left\{ \std\right\} $.
And, indeed, \eqref{eq:std almost rules of single conj class} holds
in this cases. The worst case is when $\sigma$ has $c_{1}\left(\sigma\right)=0$
and $c_{2}\left(\sigma\right)=\frac{n}{2}$, where $\left(n-2,2\right)$
rules, and $1-\x_{\left(n-2,2\right)}\left(\sigma\right)=1-\frac{1}{n-3}$,
whereas $1-\x_{\std}\left(\sigma\right)=1+\frac{1}{n-1}$. Hence one
can take $\delta_{n}=\frac{2\left(n-2\right)}{n\left(n-3\right)}$.
\end{proof}

\begin{proof}[Proof of Theorem \ref{thm:normal-sets-positive-result}]
By Lemma \ref{lem:evalues of normal sets}, for any non-negative
normal element $\Sigma\in\mathbb{R}\left[S_{n}\right]$, the eigenvalues
associated with $\rho\in\widehat{S_{n}}$ are all equal to $\sum_{C}\alpha_{C}\left|C\right|\x_{\rho}\left(C\right)$.
The gap between this eigenvalue and $\left|\Sigma\right|$ is
\[
\left|\Sigma\right|-\sum_{C}\alpha_{C}\left|C\right|\x_{\rho}\left(C\right)=\sum_{C}\alpha_{C}\left|C\right|\left(1-\x_{\rho}\left(C\right)\right).
\]
By Lemma \ref{lem:std almost rules for any single conj class}, if
$n\ge N_{0}$, this is at least
\[
\sum_{C}\alpha_{C}\left|C\right|\left[1-\x_{\std}\left(\sigma\right)\right]\cdot\left(1-\delta_{n}\right)=\left[\left|\Sigma\right|-\lambda_{1}\left(\std,\Sigma\right)\right]\cdot\left(1-\delta_{n}\right).\qedhere
\]
\end{proof}

\section{A Conjecture for Arbitrary Symmetric Sets and Its Consequences\label{sec:A-Conjecture}}

This section gives more justification, motivation and background to
Conjecture \ref{conj:kozma-puder}.

\subsection*{Background and Examples}

Recall that Conjecture \ref{conj:kozma-puder} says that there is
some $k\geq4$ such that the (signed and unsigned) actions of $S_{n}$
on $k$-tuples determine the spectral gap $\left|\Sigma\right|-\lambda\left(S_{n},\Sigma\right)$
up to a universal multiplicative factor, for every large enough $n$
and every symmetric non-negative element $\Sigma\in\mathbb{R}\left[S_{n}\right]$.
In terms of Definition \ref{def:family of irreps}, this means that
there is a finite number of families of irreps which are nearly dominant
spectrally, in the same sense. In fact, even the following stronger
version of Conjecture \ref{conj:kozma-puder} is conceivable:

\begin{ques}\label{question:stronger-conj}Is it true that for every
symmetric non-negative $\Sigma\in\mathbb{R}\left[S_{n}\right]$, the
spectral gap of the Cayley graph $\cay\left(S_{n},\Sigma\right)$
is equal, up to a decaying multiplicative factor, to the spectral
gap given by the action on $4$-tuples and the signed action on $4$-tuples,
namely, that

\[
\left|\Sigma\right|-\lambda\left(S_{n},\Sigma\right)\ge\left[\left|\Sigma\right|-\max\left(\lambda\left(4,\Sigma\right),\lambda^{\sgn}\left(4,\Sigma\right)\right)\right]\cdot\left[1-o_{n}\left(1\right)\right]?
\]
Perhaps even the non-signed action suffices, namely, is it true that
\[
\left|\Sigma\right|-\lambda\left(S_{n},\Sigma\right)\ge\left[\left|\Sigma\right|-\lambda\left(4,\Sigma\right)\right]\cdot\left[1-o_{n}\left(1\right)\right]?
\]

\end{ques}

Note that this question differs from Conjecture \ref{conj:kozma-puder}
both by giving a specific set of irreps, and by suggesting that the
multiplicative factor tends to $1$ as $n\to\infty$. The evidence
we gathered so far supports Conjecture \ref{conj:kozma-puder} and
even its stronger version -- Question \ref{question:stronger-conj}.
The conjecture is certainly true, and with $\std$ alone, for normal
elements (Theorem \ref{thm:normal-sets-positive-result}) and for
elements supported on transpositions (Theorem \ref{thm:Aldous}).
The example of non-generating sets, where the spectral gap is zero
(see Section \ref{sec:Introduction}), shows that $\std$ alone is
not sufficient. The latter is also demonstrated by the following example
relating to the generating set consisting of an $n$-cycle and a sole
transposition:
\begin{example}
\label{exa:n-cycle and (12)}Let 
\[
\Sigma={\textstyle \frac{1}{4}}\left[\mathrm{id}+\left(1~2\right)+\left(1~2\ldots n\right)+\left(1~2\ldots n\right)^{-1}\right]\in\mathbb{R}\left[S_{n}\right].
\]
In this case, one of $\left(n-2,2\right)$ or $\left(n-2,1,1\right)$
rules, at least up to a multiplicative constant factor on the spectral
gap, whereas $1-\lambda_{1}\left(\std\left(\Sigma\right)\right)$
is of a different order. More concretely, it is known \cite[Example 1, Page 2139]{diaconis1993comparison}
that 
\[
1-\lambda\left(S_{n},\Sigma\right)\ge\frac{1}{18n^{3}},
\]
and that $\frac{1}{n^{3}}$ is the right order of the spectral gap.
The proof of the upper bound for the spectral gap in \cite{diaconis1993comparison}
can be adapted as follows: Consider $\schk 2$, and let $f$ be a
function on its vertices given by $f\left(\left(x,y\right)\right)=\left(x-y\mod n\right)-\frac{n}{2}$.
As $f$ is orthogonal to the constant functions, its Rayleigh quotient
$\frac{\left\langle \left(I-A\right)f,f\right\rangle }{\left\langle f,f\right\rangle }$,
which roughly equals $\frac{6}{n^{3}}$, gives a lower bound on the
spectral gap of this Schreier graph, where $A$ is the adjacency operator
of the graph. The adjacency operator of this graph decomposes into
the irreps $\left(n-2,1,1\right)$, $\left(n-2,2\right)$, $\std=\left(n-1,1\right)$
and $\triv=\left(n\right)$, so the second eigenvalue comes from one
of $\left(n-2,1,1\right)$, $\left(n-2,2\right)$ or $\left(n-1,1\right)$.
However, $\left(n-1,1\right)$ is not possible because $1-\lambda_{1}\left(\std\left(\Sigma\right)\right)$
is of order $\frac{1}{n^{2}}$: this is the spectral gap of the connected
$4$-regular graph $\schk{}$, and it follows from the discrete Cheeger
inequality \cite[Thm.\ 4.11]{hoory2006expander} that $1-\lambda_{1}\left(\std\left(\Sigma\right)\right)\ge\frac{1}{32n^{2}}$.
(It can be shown, in fact, that $n^{2}\cdot\left(1-\lambda_{1}\left(\std\left(\Sigma\right)\right)\right)\overset{{\scriptscriptstyle n\rightarrow\infty}}{\longrightarrow}1$.)
Simulations for small values of $n$ suggest that $\left(n-2,1,1\right)$
rules for this $\Sigma$, but we do not know whether this is true
for all $n$.
\end{example}

Let us also mention another attempt to generalize Theorem \ref{thm:Aldous}
(Aldous' conjecture) which is attributed to Caputo in \cite[Page 301]{cesi2016few}
and is named there ``the $\alpha$-shuffles conjecture''. For $A\subseteq\left[n\right]=\left\{ 1,\ldots,n\right\} $,
let $J_{A}=\sum_{\sigma\in S_{n}~\mathrm{s.t.}~\mathrm{supp}\left(\sigma\right)\subseteq A}\sigma\in\mathbb{R}\left[S_{n}\right]$.
The conjecture states that for every linear combination with non-negative
coefficients of the $J_{A}$'s
\[
\Sigma=\sum\nolimits _{A\subseteq\left[n\right]}\alpha_{A}\cdot J_{A},~~~~\alpha_{A}\ge0
\]
the standard representation rules, namely, $\lambda\left(S_{n},\Sigma\right)=\lambda_{1}\left(\std\left(\Sigma\right)\right)$.
Some special cases of this conjecture were proven in work in progress
by Gil Alon, Gady Kozma and the second author.

\subsection*{Consequence 1: Random pairs of permutations expand}

It is well known that if $g,h\in S_{n}$ are chosen uniformly and
independently at random then they generate $A_{n}$ or $S_{n}$ with
probability tending to $1$ as $n\to\infty$ \cite{dixon1969probability}.
So with high probability $\lambda\left(S_{n},\frac{1}{4}\left(g+g^{-1}+h+h^{-1}\right)\right)<1$.
But do random pairs of $S_{n}$ also generate a family of expander
Cayley graphs? Namely,

\begin{ques}\label{ques:pairs expand}Is there some $\varepsilon>0$
so that for uniformly random $g,h\in S_{n}$ 
\begin{equation}
\mathrm{Prob}\left[1-\lambda\left(S_{n},{\textstyle \frac{1}{4}}\left(g+g^{-1}+h+h^{-1}\right)\right)\geq\varepsilon\right]\stackrel{n\to\infty}{\longrightarrow}1?\label{eq:random cay graphs of Sn}
\end{equation}

\end{ques}

This question is asked for all finite non abelian simple groups in
\cite[Problem 2.28]{lubotzky2012expander} (with uniform $\varepsilon$
for all), and is proven in \cite{breuillard2015expansion} for finite
simple group of Lie type and bounded Lie rank. However, the case of
$A_{n}$ remains wide open. The best known bound is given in \cite[Theorem 1.2]{helfgott2015random},
where \eqref{eq:random cay graphs of Sn} is proven with $\varepsilon$
replaced with $\frac{c_{1}}{n^{3}\left(\log n\right)^{c_{2}}}$ for
some absolute constants $c_{1},c_{2}$. In fact, it was a major challenge
to show that $A_{n}$ (or $S_{n}$) can even be turned into an expanding
family -- this was done in \cite{kassabov2007symmetric}, and it
is still not known whether two generators suffice for this goal (see
also \cite[Problem 3.2(a)]{breuillard2018expansion}).

Conjecture \ref{conj:kozma-puder}, if true, would yield a positive
answer to Question \ref{ques:pairs expand}. This is based on the
fact, proven in \cite{friedman1998action}, that for every fixed $k$
and $r$, the Schreier graphs depicting the action of $S_{n}$ on
$k$-tuples of different elements in $\left[n\right]$ with respect
to $r$ random permutations form an expander family with high probability,
namely, there is some $\varepsilon_{k}>0$ so that\footnote{In fact, there is evidence to that for every fixed $k$, the Schreier
graph $\schk k$ with $\Sigma$ symmetric and random of fixed size
$r\ge3$, is nearly Ramanujan with high probability, namely, that
for every $\varepsilon>0$, $\lambda\left(k,\Sigma\right)<2\sqrt{r-1}+\varepsilon$
with probability tending to $1$ as $n\to\infty$. This is true for
$k=1$ \cite{Fri08} and for $k=2$ \cite{bordenave2019eigenvalues},
and it seems plausible that the latter proof may be extended to every
fixed value of $k$.} 
\[
\mathrm{Prob}\left[1-\lambda\left(k,{\textstyle \frac{1}{4}}\left(g+g^{-1}+h+h^{-1}\right)\right)\ge\varepsilon_{k}\right]\stackrel{n\to\infty}{\longrightarrow}1.
\]
A small adaptation of the argument in \cite{friedman1998action}
allows one to prove an analogous statement for the \emph{signed} action
on $\ell$-tuples: introducing a random sign to the action on $\ell$-tuples
only adds more randomness and thus can only increase the expected
spectral gap. This explains how Conjecture \ref{conj:kozma-puder},
if true, yields a positive answer to Question \ref{ques:pairs expand}.
\begin{rem}
A weaker version of Conjecture \ref{conj:kozma-puder} states that
in order to approximate the spectral gap $\left|\Sigma\right|-\lambda\left(S_{n},\Sigma\right)$,
it is enough to consider, for every $n$, a set ${\cal L}_{n}\subseteq\widehat{S_{n}}$
of irreps with at most $\ell_{n}$ blocks outside the first row/column,
where $\ell_{n}$ grows slowly with $n$ (say, $\ell_{n}\sim c\cdot\log n$).
This may still be enough to yield a positive answer to Question \ref{ques:pairs expand}.
We remark that the analog result for the mere existence of a spectral
gap, namely, that $S_{n}$ has no $3\log n$-transitive subgroups
but itself and $A_{n}$, was proved by Jordan in 1895 \cite{jordan1895limite}.
Unlike the fact mentioned in Section \ref{sec:Introduction} that
there are no $4$-transitive subgroups of $S_{n}$ other than $A_{n}$
and $S_{n}$ (for $n\ge25$), this weaker result of Jordan has elementary
proofs, which, in particular, do not depend on the classification
of finite simple groups -- see \cite{babai1987degree} and the references
therein.
\end{rem}

\subsection*{Consequence 2: A bound on the diameter of Cayley graphs of $A_{n}$}

In \cite[Conjecture 1.7]{babai1992diameter}, Babai conjectured that
for every finite simple group $G$ and every generating set $\Sigma\subset G$,
\[
\mathrm{diam}\left(\cay\left(G,\Sigma\right)\right)\le\left(\log\left|G\right|\right)^{O\left(1\right)},
\]
where the implied constant is absolute. The special case of $G=A_{n}$
is referred to as a ``folklore'' conjecture. In this case, the conjecture
translates to that $\mathrm{diam}\left(\cay\left(A_{n},\Sigma\right)\right)\le n^{O\left(1\right)}$.
The best upper bound to date is quasi-polynomial and is due to Helfgott
and Seress \cite{helfgott2014diameter}.

Consider $\mathrm{Sch}\left(A_{n}\curvearrowright\smash{\left[n\right]_{k}},\Sigma\right)$,
the Schreier graph of $A_{n}$ on $k$-tuples with respect to $\Sigma\subset A_{n}$.
This graph has $n\left(n-1\right)\cdots\left(n-k+1\right)\le n^{k}$
vertices, and if $\Sigma$ is generating, it is connected. By \cite{aksoy2018maximum},
the spectral gap of the simple random walk on this graph is bounded
from below by $\left(1+o_{n}\left(1\right)\right)$$\frac{54}{n^{3k}}$
(slightly weaker bounds were known before -- see the references in
\cite{aksoy2018maximum}). If conjecture \ref{conj:kozma-puder} holds,
this gives a lower bound of the same order on the spectral gap of
the simple random walk on the Cayley graph $\cay\left(A_{n},\Sigma\right)$.
Finally, a lower bound of this kind on the spectral gap yields a polynomial
upper bound on the diameter of the Cayley graph: for instance, \cite[Equation (6.6)]{saloff2004random}
says that if $\left(1-\lambda_{2}\right)$ is the spectral gap of
the simple random walk on a Cayley graph of a finite group $G$, then
the diameter of this graph is at most 
\[
\frac{3\log\left|G\right|}{\sqrt{1-\lambda_{2}}}.
\]

\paragraph*{Acknowledgments}

We are grateful to Gady Kozma for highly valuable discussions. We
also thank Gil Alon and Avi Wigderson for beneficial comments. D.P.\ learned
about Aldous' conjecture and its proof by Caputo et al.\ in a mini-course
by Kozma as part of a summer school in Budapest in 2014. We thank
the organizers of the school, Miklós Abért, Ágnes Backhausz, Lászlo
Lóvász, Balázs Szegedy and Bálint Virág, for this wonderful event.
Computer simulations mentioned in this work were carried out using
Sage. The research was supported by the Israel Science Foundation,
ISF grant 1031/17 of O.P.\ and ISF grant 1071/16 of D.P.\bibliographystyle{alpha}
\bibliography{normal-sets}

\newcommand{\etalchar}[1]{$^{#1}$}
\begin{thebibliography}{BGGT15}

\bibitem[ACTT18]{aksoy2018maximum}
Sinan~G Aksoy, Fan Chung, Michael Tait, and Josh Tobin.
\newblock The maximum relaxation time of a random walk.
\newblock {\em Advances in Applied Mathematics}, 101:1--14, 2018.

\bibitem[BC19]{bordenave2019eigenvalues}
Charles Bordenave and Beno{\^\i}t Collins.
\newblock Eigenvalues of random lifts and polynomials of random permutation
  matrices.
\newblock {\em Annals of Mathematics}, 190(3):811--875, 2019.

\bibitem[BGGT15]{breuillard2015expansion}
Emmanuel Breuillard, Ben~J. Green, Robert~M. Guralnick, and Terence Tao.
\newblock Expansion in finite simple groups of {L}ie type.
\newblock {\em Journal of the European Mathematical Society}, 17(6):1367--1434,
  2015.

\bibitem[BL18]{breuillard2018expansion}
Emmanuel Breuillard and Alexander Lubotzky.
\newblock Expansion in simple groups.
\newblock arXiv:1807.03879, 2018.

\bibitem[BS87]{babai1987degree}
L{\'a}szl{\'o} Babai and {\'A}kos Seress.
\newblock On the degree of transitivity of permutation groups: a short proof.
\newblock {\em Journal of Combinatorial Theory, Series A}, 45(2):310--315,
  1987.

\bibitem[BS92]{babai1992diameter}
L{\'a}szl{\'o} Babai and {\'A}kos Seress.
\newblock On the diameter of permutation groups.
\newblock {\em European journal of combinatorics}, 13(4):231--243, 1992.

\bibitem[Cam99]{cameron1999permutation}
Peter~J. Cameron.
\newblock {\em Permutation groups}, volume~45 of {\em Student Texts}.
\newblock Cambridge University Press, 1999.

\bibitem[Ces16]{cesi2016few}
Filippo Cesi.
\newblock A few remarks on the octopus inequality and {A}ldous' spectral gap
  conjecture.
\newblock {\em Communications in Algebra}, 44(1):279--302, 2016.

\bibitem[Ces20]{cesi2020spectral}
Filippo Cesi.
\newblock On the spectral gap of some {C}ayley graphs on the {W}eyl group
  {$W(B_n)$}.
\newblock {\em Linear Algebra and its Applications}, 586:274--295, 2020.

\bibitem[CLR10]{caputo2010proof}
Pietro Caputo, Thomas Liggett, and Thomas Richthammer.
\newblock Proof of {A}ldous' spectral gap conjecture.
\newblock {\em Journal of the American Mathematical Society}, 23(3):831--851,
  2010.

\bibitem[Dia88]{diaconis1988group}
Persi Diaconis.
\newblock Group representations in probability and statistics.
\newblock {\em Lecture Notes-Monograph Series}, 11:i--192, 1988.

\bibitem[Dix69]{dixon1969probability}
John~D. Dixon.
\newblock The probability of generating the symmetric group.
\newblock {\em Mathematische Zeitschrift}, 110(3):199--205, 1969.

\bibitem[DSC93]{diaconis1993comparison}
Persi Diaconis and Laurent Saloff-Coste.
\newblock Comparison techniques for random walk on finite groups.
\newblock {\em The Annals of Probability}, pages 2131--2156, 1993.

\bibitem[FH91]{fulton1991representation}
William Fulton and Joe Harris.
\newblock {\em Representation theory: a first course}.
\newblock Springer, 1991.

\bibitem[FJR{\etalchar{+}}98]{friedman1998action}
Joel Friedman, Antoine Joux, Yuval Roichman, Jacques Stern, and Jean-Pierre
  Tillich.
\newblock The action of a few permutations on r-tuples is quickly transitive.
\newblock {\em Random Structures and Algorithms}, 12(4):335--350, 1998.

\bibitem[Fri08]{Fri08}
J.~Friedman.
\newblock {\em A proof of Alon's second eigenvalue conjecture and related
  problems}, volume 195 of {\em Memoirs of the AMS}.
\newblock AMS, september 2008.

\bibitem[Ful97]{fulton1997applications}
William Fulton.
\newblock Young tableaux, with applications to representation theory and
  geometry.
\newblock {\em London Mathematical Society Student Texts}, 35, 1997.

\bibitem[Gre20]{Greenhut2020}
Ziv Greenhut.
\newblock A generation criterion for subsets of special linear groups over
  finite fields.
\newblock arXiv preprint arXiv:2002.06461, 2020.

\bibitem[HLW06]{hoory2006expander}
Shlomo Hoory, Nathan Linial, and Avi Wigderson.
\newblock Expander graphs and their applications.
\newblock {\em Bulletin of the American Mathematical Society}, 43(4):439--561,
  2006.

\bibitem[HS14]{helfgott2014diameter}
Harald~A. Helfgott and Akos Seress.
\newblock On the diameter of permutation groups.
\newblock {\em Annals of mathematics}, 179(2):611--658, 2014.

\bibitem[HSZ15]{helfgott2015random}
Harald~A. Helfgott, {\'A}kos Seress, and Andrzej Zuk.
\newblock Random generators of the symmetric group: diameter, mixing time and
  spectral gap.
\newblock {\em Journal of Algebra}, 421:349--368, 2015.

\bibitem[Jor95]{jordan1895limite}
Camille Jordan.
\newblock Nouvelles recherches sur la limite de transitivit{\'e} des groupes
  qui ne contiennent pas le groupe altern{\'e}.
\newblock {\em J. de Math. Pures et Appliqu{\'e}es}, 1:35--60, 1895.

\bibitem[Kas07]{kassabov2007symmetric}
Martin Kassabov.
\newblock Symmetric groups and expander graphs.
\newblock {\em Inventiones mathematicae}, 170(2):327--354, 2007.

\bibitem[LS08]{larsen2008characters}
Michael Larsen and Aner Shalev.
\newblock Characters of symmetric groups: sharp bounds and applications.
\newblock {\em Inventiones mathematicae}, 174(3):645, 2008.

\bibitem[Lub12]{lubotzky2012expander}
Alexander Lubotzky.
\newblock Expander graphs in pure and applied mathematics.
\newblock {\em Bulletin of the American Mathematical Society}, 49(1):113--162,
  2012.

\bibitem[Mac95]{macdonald1979symmetric}
I.G. Macdonald.
\newblock {\em Symmetric functions and {H}all polynomials}.
\newblock Oxford U., 1995.

\bibitem[SC04]{saloff2004random}
Laurent Saloff-Coste.
\newblock Random walks on finite groups.
\newblock In {\em Probability on discrete structures}. Springer, 2004.

\end{thebibliography}

\noindent \noun{Ori Parzanchevksi, Einstein School of Mathematics,
The Hebrew University}\\
\texttt{parzan@math.huji.ac.il\medskip{}
}

\noindent \noun{Doron Puder, School of Mathematical Sciences, Tel
Aviv University}\\
\texttt{doronpuder@gmail.com}
\end{document}